\documentclass[12pt]{amsart}
\usepackage{amsmath, amsthm, amscd, amsfonts, amssymb, graphicx, color}
\usepackage[margin=2.8cm]{geometry}
\usepackage[bookmarksnumbered, colorlinks, plainpages]{hyperref}

\usepackage{tikz} 

\numberwithin{equation}{section}

\allowdisplaybreaks
\usepackage{mathabx}

\usepackage{mathtools}
\mathtoolsset{showonlyrefs,showmanualtags}

\usepackage{hyperref} 
\hypersetup{
    colorlinks=true,       
    linkcolor=blue,          
    citecolor=magenta,        
    filecolor=magenta,      
    urlcolor=cyan           
}

\usepackage[msc-links]{amsrefs}

\theoremstyle{plain} 
\newtheorem{lemma}[equation]{Lemma} 
\newtheorem{proposition}[equation]{Proposition} 
\newtheorem{theorem}[equation]{Theorem} 
\newtheorem{corollary}[equation]{Corollary} 

\newtheorem{definition}[equation]{Definition}

\theoremstyle{remark}

\DeclareMathOperator{\Log}{Log}

\title{Endpoint $ \ell ^{r}$ improving estimates for Prime averages}

\author[Lacey]{Michael T. Lacey}   
\address{ School of Mathematics, Georgia Institute of Technology, Atlanta GA 30332, USA}
\email {lacey@math.gatech.edt}
\thanks{MTL: The author is a 2020 Simons Fellow. Research supported in part by grant  from the US National Science Foundation, DMS-1949206}

\author[Mousavi]{Hamed Mousavi}   
\address{ School of Mathematics, Georgia Institute of Technology, Atlanta GA 30332, USA}

\author[Rahimi]{Yaghoub Rahimi}
\address{ School of Mathematics, Georgia Institute of Technology, Atlanta GA 30332, USA}

\allowdisplaybreaks
\begin{document}
\begin{abstract}  
Let $ \Lambda $ denote von Mangoldt's function, and consider the averages 
\begin{align*}
A_N f (x) &=\frac{1}{N}\sum_{1\leq n  \leq N}f(x-n)\Lambda(n) . 
\end{align*}
We prove sharp $ \ell ^{p}$-improving for these averages, and sparse bounds for the maximal function.  
The simplest inequality is that for sets $ F, G\subset [0,N]$ there holds 
\begin{equation*}
N ^{-1} \langle A_N \mathbf 1_{F} , \mathbf 1_{G} \rangle \ll  \frac{\lvert  F\rvert \cdot \lvert  G\rvert} { N ^2 }
\Bigl( \operatorname {Log} \frac{\lvert  F\rvert \cdot \lvert  G\rvert} { N ^2 } \Bigr) ^{t}, 
\end{equation*}
where $ t=2$, or assuming the Generalized Riemann Hypothesis, $ t=1$.   The corresponding sparse bound is proved 
for the maximal function  $ \sup _N A_N \mathbf 1_{F}$.  The inequalities for $ t=1$ are sharp. 
The proof depends upon the Circle Method, and an interpolation argument of Bourgain.  
\end{abstract}
\maketitle

\tableofcontents 

\section{Introduction} 

We consider discrete averages over the prime integers.  The averages are  weighted by the von Mangoldt function.  
\begin{align}
A_N f (x) &=\frac{1}{N}\sum_{1\leq n  \leq N}f(x-n)\Lambda(n) 
\\
\Lambda(n)& =\begin{cases}
\log(p) & n=p^{a}, \textup{$ p$ prime}\\
0 & \textup{Otherwise}.
\end{cases}
\end{align}
Our interest is in \emph{scale free} $ \ell ^{r}$ improving estimates for these averages.  
The question presents itself in different forms.  

For an interval $ I$ in the integers and function $ f \;:\; I \to \mathbb C $, set  
\begin{equation}\label{e:bracket}
\langle f \rangle _{I, r} = \Bigl[ \lvert  I\rvert ^{-1} \sum_{x\in I} \lvert  f (x)\rvert ^{r}   \Bigr] ^{1/r} . 
\end{equation}
If $ r=1$, we will suppress the index in the notation.  
And, set $ \operatorname {Log} x = 1 + \lvert  \log x\rvert $, for $ x>0$.

The kind of estimate we are interested in takes the   the following form, in the simplest instance.
What is the `smallest' function $ \psi \;:\; [0,1] \to [1, \infty )$ 
so that for all integers $ N$ and  indicator functions $ f , g \;:\; I \to \{0,1\}$, there holds 
\begin{equation*} 
N ^{-1} \langle A_N f, g  \rangle  \leq \langle f \rangle _{I} \langle g \rangle _{I} \psi  (\langle   f\rangle_I  \langle g \rangle_I). 
\end{equation*}
That is, the right hand side  is independent of $ N$, making it scale-free.  We specified that $ f, g$ be indicator functions as that is sometimes  the sharp form of the inequality. 
Of course it is interesting for arbitrary functions, but the bound above is not homogeneous, so not the most natural estimate in that case.  

The points of interest in these two results arises from, on the one hand, the distinguished role of the prime integers.  
And, on the other, endpoint results are significant interest in Harmonic Analysis, as the techniques which apply are the sharpest possible. 
In this instance, the sharp methods depend very much on the prime numbers. 

For the primes, we expect that the Riemann Hypothesis to be relevant.    
We state unconditional results, and those that depend upon the Generalized Riemann Hypothesis (GRH). 
Note that according to GRH all zeroes in the critical strip $0<Re(s)<1$ of an arbitrary $L-$function $L(f,s)$ are on the critical line $Re(s)=\frac{1}{2}$. 
Under GRH,  the primes are equitably distributed mod $ q$, with very good error bounds.  
Namely, 
\begin{equation} \label{e:primecount}
\psi(x,q,a)=\sum_{\substack{n<x\\ n\equiv a\pmod{q}}}\Lambda(n)=\frac{x}{\phi(q)}+O(x^{\frac{1}{2}}\log^2(q)).
\end{equation}

\begin{theorem}\label{t:first}  There is a constant $ C$ so that this holds.  
For integers $ N > 30$, and interval $ I$ of length $ N$, the following inequality holds 
for all functions $ f=   \mathbf 1_{F}$ and $  g  =  \mathbf 1_{G}$ with $ F, G\subset I$
\begin{equation}\label{e:first}
N ^{-1} \langle A_N f, g \rangle \leq C \langle f \rangle _{I} \langle g \rangle _{I} 
\times 
\begin{cases}
 \operatorname {Log} (\langle f \rangle _{I} \langle g \rangle _{I})  &  \textup{assuming GRH}
 \\
  (\operatorname {Log}    (\langle f \rangle _{I} \langle g \rangle _{I})  ) ^{t}& 
\end{cases}
\end{equation}
\end{theorem}

The inequality assuming GRH is sharp, as can be seen by taking $ f $ to be  the indicator of the primes, and $ g=\mathbf 1_{ 0}$.  
It is also desirable to have a form of the inequality  above that holds for the maximal function 
\begin{equation*}
A ^{\ast} f = \sup _{N} \lvert  A_N f\rvert.   
\end{equation*}
Our second main theorem is  sparse bound for $ A ^{\ast} $. The definition of a sparse bound is postponed to  Definition~\ref{d:sparse}.  
Remarkably, the inequality takes the same general form, although we consider a substantially larger operator.  

\begin{theorem}\label{t:sparseA} 
For functions  $   f  = \mathbf 1_{F}$ and $ g=  \mathbf 1_{G}$,  for finite sets $ F, G \subset \mathbb Z $, there is a sparse collection of intervals $ \mathcal S$ so that we have 
\begin{equation}\label{e:sparseA}
\langle A ^{\ast} f, g \rangle \lesssim \sum_{I\in \mathcal S}  \langle f \rangle _{I  }  \langle g \rangle _{I }  
(\operatorname {Log} \langle f \rangle _{I  }  \langle g \rangle _{I }   ) ^{t}
\lvert  I\rvert, 
\end{equation}
where we can take $ t=1$ under GRH, and otherwise we take $ t=2$.  
\end{theorem}

The sparse bound is very strong, implying weighted inequalities for the maximal operator $ A ^{\ast} $.  
These inequalities could be further quantified, but we do not detail those consequences, as they are essentially known.   
See \cite{MR3897012}.   
One way to see that the sparse bound is stronger is  these inequalities are a corollary.  

\begin{corollary}\label{c:logell}  The maximal operator $ A ^{\ast} $ satisfies these inequalities, where $ t=1$ under GRH, and $ t=2$ otherwise. 
First, a sparse bound with $ \ell ^{p}$ norms.  For all $ 1< p < 2$, there holds 
\begin{equation}\label{e:p-sparse}
\langle A ^{\ast} \mathbf 1_{F}, \mathbf 1_{G} \rangle \lesssim  (p-1) ^{-t} \sup _{\mathcal S}
\sum_{I\in \mathcal S}  \langle \mathbf 1_{F} \rangle _{I ,p }  \langle \mathbf 1_{G} \rangle _{I,p }  
\lvert  I\rvert. 
\end{equation}
Second, the restricted weak-type inequalities 
\begin{equation}  \label{e:RWT}
\sup _{0< \lambda < 1} \frac  \lambda { (\Log \lambda ) ^t }  \lvert   \{A ^{\ast} \mathbf 1_{F} > \lambda \}\rvert  \lesssim \lvert  F\rvert  . 
\end{equation}
Third, the weak-type inequality below holds for finitely supported non-negative functions $ f$ on $ \mathbb Z $
\begin{equation} \label{e:orlicz}
\sup _{\lambda >0} \lambda \lvert   \{A ^{\ast} f > \lambda \}\rvert \lesssim  \lVert f\rVert _{ \ell (\log \ell ) ^{t} (\log\log \ell )}
\end{equation}
where the last norm is defined in \S\ref{s:corollary}.  
\end{corollary}

This subject is an outgrowth of Bourgain's  fundamental work on arithmetic ergodic theorems \cite{MR937582,MR1019960}. 
These inequalities proved  therein focused on the diagonal case, principally $ \ell ^{p}$ to $ \ell ^{p}$ estimates for maximal functions.  
Bourgain's work has been very influential, with a very rich and sophisticated theory devoted to the diagonal estimates. 
We point to just two papers,  \cite{MR2188130}, and very recently \cite{2020arXiv200805066T}.  The subject is very rich, and the reader should consult the references in these papers.

Shortly after Bourgain's first results, Wierdl \cite{MR995574} studied the primes, and the simpler form of the Circle method in that case allowed 
him to prove diagonal inequalities for all $ p>1$, which was a novel result at that time.  
The result was revisited by Mirek and Trojan \cite{MR3370012}.   
The  unconditional version of the endpoint result \eqref{e:RWT} above is the main result of  Trojan \cite{MR4029173}. 
The approach of this paper differs in some important aspects from the one in \cite{MR4029173}. 
(The low/high decomposition is dramatically different, to point to the single largest difference.) 

The subject of sparse bounds originated in harmonic analysis, with a detailed set of applications in the survey \cite{2018arXiv181200850P}, 
with a wide set of references therein.  
The paper  \cite{MR3892403} initiated the study of  sparse bounds in the discrete setting. 
While the result in that paper of an `$ \epsilon $ improvement' nature, for averages it turns out there are very good results 
available, as was first established for the discrete sphere in \cites{MR4064582,MR4149830}. 
There is a rich theory here, with a range of inequalities for the Magyar-Stein-Wainger \cite{MR1888798} maximal function in \cite{MR4041278}. 
Nearly sharp results for certain polynomial averages  are established in \cites{2019arXiv190705734H,2020arXiv200211758D}, 
and a surprisingly good estimate for arbitrary polynomials is in \cite{MR4106792}. 
The latter result plays an interesting role in the innovative result of Krause, Mirek and Tao \cite{2020arXiv200800857K}.  

The $ \ell ^{p}$ improving property for the primes was investigated in \cite{MR4072599}, but not at the endpoint. 
That paper result established the first weighted estimates for the averages for the prime numbers.  This paper establishes the 
sharp results, under GRH.  Mirek \cite{MR3375866} addresses the diagonal case for {P}iatetski-{S}hapiro primes. 
It would be interesting to obtain $ \ell ^{p}$ improving estimates in this case. 

\smallskip 
Our proof uses the Circle Method to approximate the Fourier multiplier, following Bourgain \cite{MR937582}. 
In the unconditional case, we use Page's Theorem, which leads to the appearance of exceptional characters in the Circle method.  
Under GRH, there are no exceptional characters, and one can identify, as is well known, a very good approximation to the multiplier.

The Fourier multiplier is decomposed at the end of \S\ref{s:Approx} in such a way to fit    an interpolation argument of Bourgain \cite{MR812567}, also see \cite{MR2053347}. We call it the High/Low Frequency method. 
To acheive the endpoint results, this decomposition has to be carefully phrased.  
There are two additional features of this decomposition we found necessary to add in.  
First, certain difficulties associated with Ramanujan sums are addressed by making a significant change to a 
Low Frequency term.  
The sum defining the Low Frequency term \eqref{e:Lo} is over all $ Q$-smooth square free denominators.  
Here, the integer $ Q$ can vary widely, as small as $ 1$ and as large as $ N ^{1/10}$, say.  
(The largest $ Q$-smooth square denominator will be of the order of $ e ^{Q}$.)    
Second, in the unconditional case, the exceptional characters are grouped into their own term.  As it turns out, they can be viewed 
as part of the Low Frequency term.  
The properties we need for the High/Low method are detailed in \S\ref{s:High}. The following sections are applications of those properties. 

\section{Notation} \label{s:notation}
\section{Notation} \label{s:notation}

We write $ A \ll B$ if there is a constant $ C$ so that $ A \leq C B$.  In such instances, the exact nature of the constant is not important.

Let $ \mathcal F$ denote the Fourier transform on $ \mathbb R $, defined for by 
\begin{equation*}
\mathcal F f (\xi ) = \int _{\mathbb R } f (x) e ^{- 2 \pi i x \xi } \; dx , \qquad  f\in L ^{1} (\mathbb R ). 
\end{equation*}
The Fourier transform on $ \mathbb Z $ is denoted by $ \widehat f  $, defined by 
\begin{equation*}
\widehat f (\xi ) = \sum_{n\in \mathbb Z } f (n) e ^{- 2 \pi i n \xi }, \qquad f\in \ell ^{1} (\mathbb Z ).   
\end{equation*}

 Throughout, we denote $ A _q = \{ a \in \mathbb Z /q \mathbb Z  \;:\; (a,q)=1\}$, so that $ \lvert  A _q\rvert = \phi (q) $, the totient function.   
 We have 
 \begin{equation}\label{e:totient}
 \frac{q} {\Log\Log q } \ll \phi (q) \leq q -1.  
\end{equation}
It is known that for non-principal characters$\chi $, 
we have  $|G(\chi,a)|<q^{-\frac{1}{2}}$, see \cite{MR2061214}*{Chapter 3}. In particular, if $\chi$ is the principal character, then we get Ramanujan's sum 
\begin{align} \label{e:ramDef}
c_q(n):=\phi(q)G( \mathbf 1_{A_q},a)=\sum_{r \in A_q}e\bigl(\frac{ra}{q}\bigr).
\end{align}


Let  $ \chi _q$ denote the exceptional character. 
It is a non-trivial quadratic  Dirichlet character modulo $ q$, that is  $ \chi _q $ takes values $ -1,0,1$, and takes the value $ -1$ at least once.  
We also know that $ \chi _q$ is primitive, namely that its period is $ q$.  
As a matter of convenience, if $ q$ does not have an exceptional character, we will set $ \chi _{q} \equiv 0$, and $ \beta _q =1$.  
These properties are important to Lemma~\ref{l:exceptional}.  

Page's Theorem uses the exceptional characters to give an approximation to the prime counting function.   
Counting primes in an arithmetic progression of modulus $ q$, we have 
\begin{gather} \label{e:page}
\psi (N; q, r) - \frac{N} {\phi (q)} + \frac{\chi _q (x)} {\phi (q)} \beta ^{-1}_q  x ^{\beta_q }  \ll N e ^{ c \sqrt{\log N}    } . 
\end{gather}

\section{Approximations of the Kernel} \label{s:Approx}

Denote the kernel of $ A_N$ with the same symbol, so that $ A_N (x) = N ^{-1} \sum_{n\leq N} \Lambda (n) \delta _{n} (x)$. 
It follows that 
\begin{equation*}
\widehat {A_N} (\xi ) =  \frac{1}N\sum_{n \leq N} \Lambda (n) e ^{- 2 \pi n \xi }. 
\end{equation*}
The core of the paper is the approximation to $ \widehat {A_N} (\xi )$, and its further properties, detailed in the next section. 

 Set 
\begin{equation}\label{e:M}
M _{N} ^{\beta } = \frac{1}{N \beta } \sum_{n\leq N}  [ n ^{\beta } - (n-1) ^{\beta }] \delta _n , \qquad  \tfrac1{2} < \beta \leq 1. 
\end{equation}
We write $ M_N = M _{N} ^{1}$ when $ \beta =1$, which is the standard average. 
For $   \beta < 1$, these are not averaging operators.  They are the operators associated to the exceptional characters. 
The Fourier transforms are straight forward to estimate. 

\begin{proposition}\label{p:M}  We have the estimates 
\begin{gather}\label{e:1Hat}
\lvert  \widehat {M_N} (\xi ) \rvert  \ll  \min \{ 1,  (N \lvert  \xi \rvert ) ^{-1} \}, 
\\  \label{e:betaHat1}
\lvert  \widehat {M_N ^{\beta }} (\xi ) \rvert  \ll   (N \lvert  \xi \rvert ) ^{-1},
\\  \label{e:betaHat2}
\lvert  \widehat {M_N ^{\beta }} (\xi ) - \beta ^{-1}  N ^{\beta -1}  \rvert  \ll   N ^{\beta }\lvert  \xi \rvert . 
\end{gather}
\end{proposition}

For integers $ q$ and $ a \in A _q$, 
\begin{gather}\label{e:Laq}
\widehat {L ^{a,q}_N  } (\xi )  =   G (\mathbf 1_{A _q}, a) \widehat {M_N} (\xi ) - G (\chi _q ,a) \widehat {M ^{\beta _q} _N } (\xi   ) 
\end{gather}

We state the approximation to the kernel at rational point, with small denominator.  

\begin{lemma}\label{PNTlemma}
Assume that $|\xi-\frac{a}{q}|\leq N^{-1}Q$ for some $1\leq a\leq q\leq Q$ and $\gcd(a,q)=1$. Then 
\begin{align}  \label{e:PNTlemma}
\widehat{A_N}(\xi)=\widehat {L ^{a,q}_N  } (\xi - \tfrac{a}{q})+
\bigg\{\begin{array}{lr}
       O(QN^{-\frac{1}{2}+\epsilon}),  &  \text{ Assuming GRH}\\
        O(Qe^{-c\sqrt{n}}), &  \text{ Otherwise} 
        \end{array}
\end{align}
\end{lemma}

\begin{proof}
We proceed under GRH, and return to the unconditional case at the end of the argument.   
The key point is that we have the approximation \eqref{e:primecount} for $ \psi (N; q,r)$.  
Set $\alpha:=\xi-\frac{a}{q}$.  
Using Abel summation, we can write 
\begin{align*}
N \widehat {M_N}  (\alpha )  & = N e (\alpha N) - \sqrt{N} e (\alpha \sqrt N)    
- 2 \pi i \alpha \int _{\sqrt{N}    } ^N  e ^{t \alpha } \; dt + O (\sqrt N). 
\end{align*}
Turning to the primes, we separate out the sum below according to residue classes mod $ q$.  Since $ \xi = \frac{a}q + \alpha $, 
\begin{align}\label{defeq}
\sum_{\ell \leq N} e (\xi \ell ) \Lambda (\ell )  &= 
\sum_{\substack{0\leq r\leq q\\ \gcd(r,q)=1}}\sum_{\substack{\ell\leq N\\ \ell\equiv r \mod q}}e(\xi \ell)\Lambda(\ell)
\\
&=
\sum_{r\in A_q}e\bigl(\tfrac{ra}{q}\bigr)\sum_{\substack{\ell\leq N\\ \ell\equiv r\mod{q}}}e  (\alpha \ell  )\Lambda(\ell).
\end{align}
Examine the inner sum.  
Using Abel's summation formula, and the notation $ \psi $ for  prime counting function,  we have 
\begin{align*}
\sum_{\substack{\ell\leq N\\ \ell\equiv r\mod q }}e(\alpha \ell)\Lambda(\ell)&=\psi(N;q,r)e(\alpha N)-\psi(\sqrt{N};q,r)e(\alpha \sqrt{N})
\\ & \qquad -2\pi i\alpha\int_{\sqrt{N}}^N \psi(t;q,r)e(\alpha t)dt +O(\sqrt{N}) .
\end{align*}

At this point we can use the Generalized Riemann Hypothesis.  From \eqref{e:primecount}, it follows that 
\begin{align*}
\sum_{\substack{\ell\leq N\\ \ell\equiv r\mod q}} e(\alpha \ell)\Lambda(\ell)- \frac{N}{\phi(q)}
\widehat {M_N} (\alpha )  
& =  (\psi(N;q,r)-  \frac N  {\phi (q)} e (\alpha N)  ) e(\alpha N) 
\\ & \qquad 
- 2 \pi i     \alpha \int _{\sqrt{N}    } ^N   e (t \alpha )  (\psi(t;q,r) -t ) \; dt   + O ( \sqrt{N}    ) 
\\
&\ll N^{\frac{1}{2}+\epsilon}+\frac{Q}{N}\int_{\sqrt{N}}^N t^{\frac{1}{2}+\epsilon}dt + O (  N ^{\frac{1}2+ \epsilon })
\\ & \ll QN^{\frac{1}{2}+\epsilon}.
\end{align*}

The proof without GRH uses Page's Theorem \eqref{e:page} in place of \eqref{e:primecount}. We omit the details.  

\end{proof}

The previous Lemma approximates  $ \widehat {A_N} (\xi )$ near a rational point. 
We extend this approximation to the entire circle.  
This is done with these definitions.  
\begin{gather}
\\  \label{e::Vsn}
\widehat {V_{s,n} }(\xi )    = \sum_{\substack{a/q \in \mathcal R_s}}  G (\mathbf 1_{A_q}, a) \widehat {M_N} ( \xi -a/q)   \eta _{s} (\xi -a/q) , 
\\ \label{e::Wsn}
\widehat {W_{s,n}} (\xi )    = \sum_{a/q \in \mathcal R_s} G ( \chi _q , a) \widehat {M_N ^{\beta _q}}    (\xi -a/q ) \eta _{s} (\xi -a/q) , 
\\  
\mathcal R _{s}  = \{ a/q \;:\; a \in A _q, \   2 ^{s} \leq q < 2 ^{s+1}  \},  
\end{gather} 
and $ \mathcal R_0 = \{0\}$.  Further $ \mathbf 1_{ [-1/4,1/4]} \leq \eta \leq \mathbf 1_{[-1/2,1/2]} $, and $ \eta _s (\xi ) = \eta (4 ^{s} \xi )$.  
In \eqref{e:Wsn}, recall that if $ q$ is not exceptional, we have $ \chi _q =0$.  Otherwise, $ \chi _q$ is the associated exceptional Dirichlet character.  
Given integer $ N = 2 ^{n}$, set 
\begin{equation}\label{e:tildeN}
\tilde N  = 
\begin{cases}
e ^{c \sqrt{n}  /4   }  &  \textup{where $ c$ is as in \eqref{e:PNTlemma}}
\\
N ^{1/5}   & \textup{under GRH} 
\end{cases}
\end{equation}

\begin{lemma}\label{l:QQ} 
 Let  $ N= 2 ^{n}$.  Write $ A_N = B_N + \textup{Err}_N$, where 
\begin{equation}\label{e:BN}
B_N = \sum_{s \;:\; 2 ^{s} <   (\tilde N) ^{1/400} } V _{s,n} - W _{s,n} .
\end{equation}
Then, we have $ \lVert \textup{Err}_N f \rVert _{\ell ^2 } \ll  (\tilde N) ^{-1/1000} \lVert f\rVert _{\ell ^2 }$. 
\end{lemma}

\begin{proof}
We estimate the $ \ell ^2 $ norm by Plancherel's Theorem. That is, we bound 
\begin{equation*}
\lVert     \widehat  {A_N} -  \widehat  { B_N}  \rVert _{  L ^{\infty } (\mathbb T )}  \ll   (\tilde N) ^{-1/1000} . 
\end{equation*}

Fix $ \xi \in \mathbb T $, where we will estimate the $ L ^{\infty } $ norm above. 
By Dirichlet's Theorem, there are relatively prime integers $ a, q$ with $ 0\leq a <q \leq  (\tilde N) ^{1/5}$ with 
\begin{equation*}
 \lvert  \xi -a/q\rvert <  \frac1{q ^2 }.  
\end{equation*}
The argument now splits into cases, depending upon the size of $ q$. 

Assume that $ (\tilde N)   ^{1/400} < q \leq   (\tilde N) ^{1/5}$.  
This is a situation for which  the classical  Vinogradov inequality \cite{MR0062138}*{Chapter 9} was designed. 
That estimate is however is not enough for our purposes.  Instead we use \cite{MR2061214}*{Thm 13.6} for the estimate below. 
\begin{align*}
\lvert  \widehat  {A_N}  (\xi ) \rvert  & 
\ll  (q ^{-1/2} + (q/N) ^{1/2} + N ^{-1/5}) \log ^{3}N  \ll (\tilde N)   ^{-1/1000}. 
\end{align*}
So, in this case we should also see that $ \widehat {B_N } (\xi ) $ satisfies the same bound.  
The function $ \widehat {B_N } $ is a sum over $ \widehat {V _{s,n}} $ and $ \widehat {W _{s,n}}  $. 
The argument for both is the same.  Suppose that $\widehat {V _{s,n}} (\xi ) \neq 0 $. The supporting intervals 
for $ \eta _{s} (\xi - a/q)$ for $ a/q\in \mathcal R_s$ are pairwise disjoint.  
We must have $ \lvert  \xi - a_0/q_0\rvert < 2 ^{-2s} $ for some $ a_0/q_0 \in \mathcal R_s$, where $ 2 ^{s} < (\tilde N) ^{1/400} $. 
Then, 
\begin{equation*}
\lvert  \xi - a_0/q_0\rvert \geq  \lvert  a_0/q_0 - a/q\rvert - \lvert  \xi -a/q\rvert 
\geq (q q_0) ^{-1}  - q ^{-2} \geq q_0 ^{-4}.  
\end{equation*}
But then by the decay estimate \eqref{e:1Hat}, we have 
\begin{align*}
\lvert  G (\mathbf 1_{A_q}, a_0) \widehat {M_N} (\xi - a_0/q_0) \rvert \ll  (N q_0 ^{-4}) ^{-1} \ll  N ^{-1}   (\tilde N) ^{1/100}
\end{align*}
This estimate is summed over $ s \leq (\tilde N) ^{1/400} $ to conclude this case.  

\smallskip 

Proceed under the assumption that $ q  \leq N_0 = (\tilde N)   ^{1/400}$.  
From  Lemma~\ref{PNTlemma},  the inequality \eqref{e:PNTlemma} holds. 
\begin{align*}
\widehat{A_N}(\xi) & =\widehat {L ^{a,q}_N  } (\xi - \tfrac{a}{q})+  O (  N_0 ^{-1/2})
\end{align*}
The Big $ O$ term is as is claimed, so we verify that $ \widehat{ B_N}( \xi ) -  \widehat {L ^{a,q}_N  } (\xi - \tfrac{a}{q}) \ll  N _0 ^{-1/2}$.  

The  analysis depends upon how close $ \xi $ is to $ a/q$.  Suppose 
that $ \lvert  \xi -a/q \rvert < \tfrac{1}4  N_0 ^{-2} $.  Then $ a/q$ is the unique rational $ b/r$ with $ (b,r)=1$  and $ 0\leq b < r \leq   N_0 $ that meets this criteria. 
That means that 
\begin{align*}
   \widehat  {B_N } (\xi ) 
   &= \widehat {L ^{a,q}_N  } (\xi -  a/q)  \eta _{s} (\xi -a/q)  
\end{align*}
where in the last term on the right,  $ 2 ^{s} \leq q < 2 ^{s+1}$.  
By definition $ \eta _{s}(\xi -a/q) = \eta (4 ^{s} (\xi -a/q)  ) $, which equals one  by assumption on $ \xi $. 
  That completes this case.  

Continuing, suppose that there is no $ a/q$ with $ \lvert  \xi -a/q\rvert <   N_0 ^{-2} $. The point is that we have 
the decay estimates \eqref{e:1Hat} and \eqref{e:betaHat1} which imply 
\begin{equation*}
\lvert  \widehat {M_N} (\xi - a/q) \rvert +  \lvert  \widehat {M_N ^{\beta }} (\xi -a/q ) \rvert  \ll  [ N (\xi -a/q)] ^{-1} \ll   \frac{ N_0 ^2 } { N} 
\ll N ^{-3/5}. 
\end{equation*}
But then,  from the definition \eqref{e:Laq},  we have 
\begin{equation*}
 \lvert  \widehat {L ^{a,q}_N  } (\xi - \tfrac{a}{q})\rvert \ll   N ^{-1/5} . 
\end{equation*}
And as well, trivially bounding Gauss sums by $ 1$, we have 
\begin{equation*}
\lvert     \widehat {B_N} (\xi )  \rvert  \ll  \frac{   n ^{3/5}}N \ll   N ^{-1/5}, 
\end{equation*}
by just summing over all $ a/q \in \mathcal R _s$, with $  s <  (\tilde N) ^{1/400}  $.  
That completes the proof.  

\end{proof}

The discussion to this point is of a standard nature.  We state here a decomposition of the operator $ B_N$ 
defined in \eqref{e:BN}.  
It encodes our  High/Low/Exceptional decomposition, and requires some care to phrase, in order to prove our endpoint type 
results for the prime averages.  
It depends upon a supplementary parameter $ Q$. 
This parameter $ Q$ will play two roles, controlling the size and smoothness of denominators.  
Recall that an integer $ q$ is  \emph{$ Q$-smooth} if all of its prime factors are less than $ Q$.  
Let $ \mathbb S _Q$ be the collection of  square-free $ Q$-smooth integers.   
\begin{gather}
\\  \label{e:VsnLo}
\widehat {V_{s,n} ^{Q , \textup{lo}} } (\xi )    = \sum_{\substack{a/q \in \mathcal R_s\\ q\in \mathbb S _Q }} G (\mathbf 1_{A_q}, a) \widehat {M_N} ( \xi -a/q)   \eta _{s} (\xi -a/q) , 
\\ \label{e:VsnHi}
\widehat {V_{s,n} ^{Q , \textup{hi}} } (\xi )    = \sum_{\substack{a/q \in \mathcal R_s\\ q \not\in \mathbb S _Q }} G (\mathbf 1_{A_q}, a) \widehat {M_N} ( \xi -a/q)   \eta _{s} (\xi -a/q) , 
\\
\\ \label{e:Wsn}
\widehat {W_{s,n}} (\xi )    = \sum_{a/q \in \mathcal R_s} G ( \chi _q , a) \widehat {M_N ^{\beta _q}}    (\xi -a/q ) \eta _{s} (\xi -a/q) , 
\end{gather}

 Define 
\begin{gather}\label{e:Lo}
   {\operatorname{Lo} _{Q,N}}  = \sum_{s  }  V _{s,n} ^{ Q, \textup{lo} }  ,  
   \\  \label{e:Hi} 
 {\operatorname{Hi} _{Q,N} }  =   \sum_{s \;:\;  Q\leq  2 ^{s} \leq    (\tilde N) ^{1/400} }  V _{s,n} ^{ Q, \textup{hi} } - W _{s,n}
 \\ \label{e:Ex}
  \operatorname{Ex} _{Q, N}   =   \sum_{s \;:\;   2 ^{s} \leq   Q }  W _{s,n} 
\end{gather}  
Concerning these definitions, in the Low term \eqref{e:Lo},  there is no restriction on $ s$, but the sum only depends upon the 
finite number of square-free $ Q$-smooth numbers in $ \mathbb S _Q$.  
(Due to \eqref{e:mob}, the non-square free integers will not contribute to the sum.) 
The largest integer in $ \mathbb S _Q$ will be about $ e ^{Q}$, and the value of $ Q$ can be as big as $ \tilde N$. 
In the High term \eqref{e:Hi},  there are two parts associated with the principal and exceptional characters. 
For the principal characters, we exclude the square free $ Q$-smooth denominators which are both larger than $ Q$ 
and less than $  (\tilde N) ^{1/400}  $.  These are included in the Low term. We include all the denominators for the 
exceptional characters.  
 In the Exceptional term \eqref{e:Ex}, we just impose the restriction on the size of the denominator to be not more than $ Q$. 
 This will be part of the Low term.  

The sum of these three terms well approximates  $ B_N$. 

\begin{proposition}\label{p:Err'}  
Let $ 1\leq Q \leq \tilde N$.  
We have the estimate $ \lVert \textup{Err}'_N f\rVert _{\ell ^2 } \lesssim (\tilde N) ^{-1/2} \lVert f\rVert _{\ell ^2 }$, 
where 
\begin{equation}\label{e:Err'}
\textup{Err}'_N = {\operatorname{Lo} _{Q,N}} 
+{\operatorname{Hi} _{Q,N} }
+ \operatorname{Ex} _{N} + \textup{Err}_N  - B_N . 
\end{equation}

\end{proposition}

\begin{proof}
From \eqref{e:BN}, we see that 
\begin{equation*}
\widehat {\textup{Err}'_N } (\xi )  =  
\sum_{s \;:\; 2 ^{s} >  (\tilde N) ^{1/400}}  \widehat {V_{s,n} ^{Q , \textup{lo}} } (\xi )  
\end{equation*}
Recalling the definition of $ V_{s,n} ^{Q , \textup{lo}}$ from \eqref{e:VsnLo},  
it is straight forward to estimate this last sum in $ L ^{\infty } (\mathbb T )$, using the Gauss sum estimate 
$ G (\mathbf 1_{A_q}, a ) \ll \frac{\Log\Log q}q$.  
\end{proof}

\section{Properties of the High,  Low and Exceptional  Terms} \label{s:High}

The further properties of the High,    Low    and Exceptional terms are given here, in that order.

\subsection{The High Terms}
We have  the $ \ell ^2 $ estimates for the fixed scale, and 
and for the supremum over large scales, for the High term defined in \eqref{e:Hi}.  
Note that the supremum is larger by a logarithmic factor.  

\begin{lemma}\label{l:Hi}  We have the inequalities 
\begin{align}\label{e:HiFixed}
\lVert  \operatorname{Hi} _{Q,N} \rVert _{\ell ^2 \to \ell ^2 }  & \lesssim \frac{\log\log Q} Q , 
\\ \label{e:HiSup}
\lVert \sup _{N > Q ^2 }  \lvert   \operatorname{Hi} _{Q,N} f\rvert \rVert_2 & \lesssim \frac{\log\log Q \cdot \log Q} Q  \lVert f\rVert _{\ell ^2 }.  
\end{align}
\end{lemma}
We comment that the insertion of the $ Q$ smooth property into the definition of $  V _{s,n}  ^{ Q, \textup{hi} }$ in \eqref{e:VsnHi} 
is immaterial to this argument.  

\begin{proof} 
Below, we assume that there are no exceptional characters, as a matter of convenience as the exceptional characters are treated in 
exactly the same manner.  
For the inequality \eqref{e:HiFixed}, we have from the definition of the High term in \eqref{e:Hi}, and \eqref{e:VsnHi},
\begin{align*}
\lVert  \operatorname{Hi} _{Q,N} \rVert _{\ell ^2 \to \ell ^2 }  & = 
\lVert \widehat  { \operatorname{Hi} _{Q,N} } \rVert _{L ^{\infty } (\mathbb T )}
\\
&= 
\Bigl\lVert 
  \sum_{s \;:\;  Q \leq 2 ^{s} \leq  \tilde N} \widehat {  V _{s,n}  ^{ Q, \textup{hi} }} 
\Bigr\rVert _{L ^{\infty } (\mathbb T )}
\\
&\leq   
  \sum_{s \;:\;  Q \leq 2 ^{s} \leq  \tilde N} 
\lVert 
\widehat {  V _{s,n} ^{ Q, \textup{hi} }} 
\rVert _{L ^{\infty } (\mathbb T )} 
\\
& \leq 
 \sum_{s \;:\;  Q \leq 2 ^{s} \leq  \tilde N} 
\max _{2 ^{s} \leq q < 2 ^{s+1}}  \max _{a\in A _q } \lvert  G (\mathbf 1_{A _q}, a)\rvert  
\\
& \ll  \sum_{s \;:\;  Q \leq 2 ^{s} \leq  \tilde N} \max _{2 ^{s} \leq q < 2 ^{s+1}}  \frac1{\phi (q)} 
\\
& \ll  \sum _{s \;:\; Q\leq 2 ^{s}}  \log s \cdot  2 ^{-s} \ll \frac{\log\log Q} Q.  
\end{align*}
The first line is Plancherel, and the subsequent lines depend upon definitions, and the fact that the 
functions  below are disjointly supported.  
\begin{equation*}
 \{ \eta _{s} ( \cdot -a/q) \;:\;  2 ^{s} \leq q < 2 ^{s+1} ,\  a \in A _q  \}. 
\end{equation*}
Last of all, we use a well known lower bound  $ \phi (q) \gg q/\log\log q$.  

\medskip 
For the maximal inequality \eqref{e:HiSup}, we have an additional logarithmic term. This is direct consequence of the Bourgain 
multi-frequency inequality, stated in Lemma~\ref{l:BMF}.  
We   then have 
\begin{align*}
\lVert \sup _{N > Q ^2 }  \lvert   \operatorname{Hi} _{Q,N} f\rvert \rVert_{ \ell ^2}  
& \leq 
  \sum_{s \;:\;  Q \leq 2 ^{s}} 
\bigl\lVert 
 \sup _{N > Q ^2 }  \lvert   {  V _{s,n} ^{ Q, \textup{hi} }} f \rvert 
 \bigl\rVert _{\ell ^{2}}
 \\
 & \ll    \sum_{s \;:\;  Q \leq 2 ^{s}  }  
 s  \cdot  \max _{2 ^{s} \leq q < 2 ^{s+1}}  \frac1{\phi (q)}  \cdot \lVert f\rVert _{\ell ^{2}} \lesssim \frac{\log Q \cdot \log\log Q} Q  \lVert f\rVert _{\ell ^{2}}.  
\end{align*}

\end{proof}

\begin{lemma}\label{l:BMF}  Let $ \theta _1 ,\dotsc, \theta _J$ be points in $ \mathbb T $ with 
$ \min _{j\neq k} \lvert  \theta _j - \theta _k\rvert >  2 ^{-2 s_0+2} $.    We have the inequality 
\begin{equation*}
\Bigl\lVert 
\sup _{N > 4 ^{s_0}} 
\Bigl\lvert 
\sum_{j=1} ^{J}  \mathcal F ^{-1}  \Bigl(  \widehat f  \sum_{j=1} ^{J} \tilde M_N ( \cdot - \theta _j) \eta _{s_0} ( \cdot - a/q)  \Bigr)
\Bigr\rvert
\Bigr\rVert _{\ell ^{2}} \ll  \log J \cdot  \lVert f\rVert _{\ell ^{2}}. 
\end{equation*}

\end{lemma}

This is one of the main results of \cite{MR1019960}. It is stated therein with a higher power of $ \log J $. 
But it is well known that the inequality holds with a single power of $ \log J$. This is discussed in detail in \cite{MR4072599}.

\subsection{The Low Terms}

From the Low terms defined in \eqref{e:Lo},  the property is

\begin{lemma}\label{l:Lo} 
For a functions $ f ,g$ supported on  interval $ I$ of length $ N= 2 ^{n}$, 
we have 
\begin{equation}\label{e:LoLess} 
 N ^{-1} \langle  \operatorname {Lo} _{Q, N} \ast f , g \rangle  \ll    \log Q    \cdot   \langle f \rangle _{I} \langle g \rangle_I. 
\end{equation}
\end{lemma}

The following M\"obius Lemma is well known.  

\begin{lemma}\label{l:mob} For each $ q$, we have 
\begin{equation}\label{e:mob}
\sum_{a\in A _q}  G (\mathbf 1_{A_q} , a) \mathcal F ^{-1}  (  \widehat M _{N } \cdot  \eta _s   ( \cdot  - a/q)  ) (x)   
=  \frac { \mu (q)}{ \phi (q)}  c_q (-x) .
\end{equation}
\end{lemma}

\begin{proof}
Compute 
\begin{align*}
\sum_{a\in A _q}  G (\mathbf 1_{A_q} , a) \mathcal F ^{-1}  (  \widehat M _{N } \cdot  \eta _s   ( \cdot  - a/q)  ) (x)  
&= M _{N}   \ast \mathcal F ^{-1} \eta _s (x)  \sum_{a\in A _q}  G (\mathbf 1_{A_q} , a)  e (ax/q) . 
\end{align*}
We focus on the last sum above, namely 
\begin{align} \label{e:sumequals}
S_q (x) &=  \sum_{a\in A _q}  G (\mathbf 1_{q} ,a) e ( x a/q ) 
\\
& = \frac1{ \phi (q)} \sum_{r\in A _q} \sum_{a\in A _q}  e (a (r + x)/q)  
\\
&=  \frac1{ \phi (q)} \sum_{r\in A _q} c_q (r+ x) =  \frac { \mu (q)}{ \phi (q)}  c_q (-x) . 
\end{align}
The last line uses Cohen's identity.  
\end{proof}

The two steps of  inserting of the property of being $ Q$ smooth in \eqref{e:VsnLo}, as well as dropping an restriction on $ s$ in \eqref{e:Lo}, 
were made for this proof.

\begin{proof}[Proof of Lemma~\ref{l:Lo}] 
By \eqref{e:mob}, the kernel of  the operator $ \operatorname {Lo} _{Q, N}$ is 
\begin{align}
\operatorname {Lo} _{Q, N} (x)  & = M _{N}   \ast \mathcal F ^{-1} \eta _s (x) \cdot S (-x), 
\\
\textup{where} \quad 
S (x) &= \sum_{q \in \mathbb S _Q}   \frac { \mu (q)}{ \phi (q)}  c_q (x)   . 
\end{align}
We establish a pointwise bound $ \lVert S\rVert _{\ell ^{\infty }} \ll \log Q$, which proves the Lemma.  

Assume $ x \neq 0$. We exploit the  multiplicative properties of the summands, as well as the fact that if 
prime $ p$ divides $ x$, we have $ \frac{\mu _p (x)} {\phi (p)} c_q (x) = \mu _p (x)$.  
Let $ \mathcal Q_1$ be the primes $ p< Q$ such that $ (p,x)=1$, 
and set $ \mathcal Q_2$ to be the primes less than $ Q$ which are not in $ \mathcal Q_1$.

The multiplicative aspect of the sums allows us to write 
\begin{equation*}
 \frac { \mu (q)}{ \phi (q)}  c_q (-x)  =   \frac { \mu (q_1)}{ \phi (q_1)}  c_{q_1} (-x) \cdot \mu (q_2)
\end{equation*}
where $ q=q_1 q_2$, and all prime factors of $ q_j$ are in $ \mathcal Q_j$. 
If $ \mathcal Q_j$ is empty, set $ q_j =1$. 
Thus, 
$ S (x) = S_1 (x) S_2 (x)$, where the two terms are associated with $ \mathcal Q_1$ and $ \mathcal Q_2$ respectively. 
We have 
\begin{align*}
S_1 (x)  & = \sum_{ \textup{ $ q$ is $ \mathcal Q_1$ smooth}}  \frac { \mu (q)}{ \phi (q)}  c_{q} (-x) 
\\
& =   \prod _{ p \in \mathcal Q_1 }  1+ \frac{\mu (p) c_p (-x)} {\phi (p)}
\\ &=   \prod _{ p \in \mathcal Q_1 }  1+ \frac{1} {p-1}=  A_x . 
\end{align*}
This is so, since $ \mu (p)c_p (x)=1$.  
It is a straight forward consequence of the Prime Number Theorem that $ A_x \ll  \log Q$.   
Here, and below, we say that $ q$ is $ \mathcal Q$ smooth if all the prime factors of $ q$ are in the set of primes $ \mathcal Q$.

The second term is as below, where $ d = \lvert  \mathcal Q_2\rvert $. 
Here, in the definition \eqref{e:Lo}, there is no restriction on $ s$, hence all the smooth square free numbers are included.  
If $ \mathcal Q_2 = \emptyset $,  then  $ S_2 (x) =1$, otherwise 
\begin{align*}
S_2 (x) & = \sum_{ \textup{ $ q$ is $ \mathcal Q_2$ smooth}} \mu (q)  
\\
& = \sum_{j=1} ^{d}  \binom d j (-1) ^{j} 
\\
&= -1 + \sum_{j=0} ^{d}  \binom d j (-1) ^{j} = -1.  
\end{align*}
If $ x=0$, then $ S (0) = S_2 (x)=-1$.   That completes the proof.  

\end{proof}

\subsection{The Exceptional Term}

The Exceptional terms are always of a smaller order than the Low terms.

\begin{lemma}\label{l:exceptional}  
Let $ \chi $ be an exceptional character modulo $ q$.  
For $ x\in \mathbb Z $, 
\begin{equation} \label{e:exceptional}
\Bigl\lvert 
\sum_{a\in A _q} G (\chi , a) e (xa/q) 
\Bigr\rvert=  \frac{q}{\phi(q)} 
\end{equation}
provided $ (x,q)=1$, otherwise the sum is zero. 
\end{lemma}

\begin{proof}
It is also known that exceptional characters are primitive - see \cite{MR2061214}*{Theorem 5.27}. So  the sum is zero if $ (x,q)>1$. 
We use the common notation 
\begin{equation*}
\tau (\chi ,x ) = \sum_{a\in A_q} \chi (a) e (ax/q) 
\end{equation*}
which is $ \phi (q) G (\chi , x) $.  
Assuming $ (x,q)=1$, 
\begin{equation}
\tau  (\chi , a) =   \tau (\chi ,1).  
\end{equation}
This leads immediately to 
\begin{align}
\sum_{a \in A_q} \tau (\chi,a)e(\frac{ax}{q}) &=     \tau  (\chi , 1)     \sum_{a \in A_q}\chi  (a)e(-\frac{ax}{q})
\\&= \frac {\tau (\chi  )\overline{\tau(\chi, x)}}  {\phi (q)}=\frac{|\tau(\chi)|^2\overline{\chi(x)}}{\phi(q)}.
\end{align}
It is known that $|\tau(\chi)|^2=q$ for primitive characters.  And the exceptional character is quadratic, so this completes the proof.  
\end{proof}

\begin{lemma}\label{l:Ex} 
For a function $ f$ supported on  interval $ I$ of length $ N= 2 ^{n}$, 
we have 
\begin{equation}\label{e:ExLess} 
 \langle  \operatorname {Ex} _{Q,N} \ast f \rangle _{\infty } \ll    (\log\log Q ) ^2    \cdot   \langle f \rangle _{I}. 
\end{equation}
The term on the left is defined in \eqref{e:Ex}.  

\end{lemma}

\begin{proof}  Following the argument from Lemma~\ref{l:Lo}, we have 
\begin{align} \label{e:EX=}
\operatorname {Ex} _{Q,N} (x)  & =  \sum_{ q <Q}
\sum _{a\in A_{q}} G ( \chi _{q}, a)  e (xa/q)   \cdot M_ N^{\beta_v } \ast \mathcal F ^{-1} \eta _{s_q} (x) . 
\end{align}
Above, $ 2 ^{s_q} \leq q < 2 ^{s_q +1}$.  
The interior sum above is  estimated in \eqref{e:exceptional}. 
Using the lower bound on the totient function in \eqref{e:totient},  we have 
\begin{equation*}
\operatorname{Ex} _{Q,N} (x) f \ll \log\log Q \cdot   \langle f \rangle_I   \sum_{ \substack{ q <Q\\ \textup{$ q$ exceptional} }} 1   .  
\end{equation*}
We know that the exceptional $ q$ grow at the rate of a double exponential, that is for $ q_v$ being the $ v$th exceptional $ q$, 
we have $ q _{v} \gg C ^{C ^{v}}$, for some $ C>1$. It follows that the sum above is at most $ \log\log Q$.   
\end{proof}

\section{Proofs of the Fixed Scale and Sparse Bounds} 

\begin{proof}[Proof of Theorem~\ref{t:first}]   
Let $ N= 2 ^{n}$, and recall that $ f = \mathbf 1_{F}$ and $ g= \mathbf 1_{G}$ 
where $ F, G\subset I$, and interval of length $ N$. 

Let us address the case in which we do not assume GRH.  We always have the estimate 
\begin{equation} \label{e:always}
N ^{-1} \langle A_N f, g \rangle \lesssim n \cdot \langle  f\rangle _{I} \langle g \rangle _{I}. 
\end{equation}
Hence,  if we have $ \langle  f\rangle _{I} \langle g \rangle _{I} \ll e ^{- c \sqrt{n} /100   }$, 
the inequality with a squared  log  follows.  

We assume that $ e ^{- c \sqrt{n}    } \ll \langle  f\rangle _{I} \langle g \rangle _{I}$, and then prove a better estimate. 
We turn to the Low/High/Exceptional decomposition in \eqref{e:Lo}---\eqref{e:Ex},  for a choice of integer $ Q$ that we will specify.  
We have 
\begin{equation*}
 {A_N} =   {\operatorname{Lo} _{Q,N}} + {\operatorname{Hi} _{Q,N}}    -  \operatorname{Ex} _{Q,N}  + \textup{Err}_N +  \textup{Err}_N ' 
\end{equation*}
These terms are defined \eqref{e:Lo}, \eqref{e:Hi}, \eqref{e:Ex}, \eqref{e:BN} and \eqref{e:Err'} df respectively.  

For the `High' term we have by \eqref{e:HiFixed}, 
\begin{align*}
N ^{-1} \lvert  \langle \operatorname {Hi} _{Q,N} f, g\rangle \rvert  \lesssim  \frac{\log\log Q} Q \langle f \rangle _{I,2} \langle g \rangle _{I,2} 
\end{align*}
The same inequality holds for both $  \operatorname {Err} _{Q,N}f$ and $\operatorname {Err}' _{Q,N}f $ by Lemma~\ref{l:QQ} and Proposition~\ref{p:Err'}.

Concering the Low term,  by \eqref{e:LoLess}, we have 
\begin{equation*}
N ^{-1} \lvert  \langle \operatorname {Lo} _{Q,N} f, g \rangle\rvert  \lesssim  \log Q \langle f \rangle _{I} \langle g \rangle _{I} 
\end{equation*}
The Exceptional term satisfies the same estimate by \eqref{e:ExLess}.  

Combining estimates,  choose $ Q$ to minimize the right hand side, namely 
\begin{equation}  \label{e:chooseQ}
N ^{-1} \langle A_N f, g \rangle \lesssim   \frac{\log\log Q} Q \bigl[ \langle  f\rangle _{I} \langle g \rangle _{I} \bigr] ^{1/2} 
+ \log Q \cdot  \langle  f\rangle _{I} \langle g \rangle _{I}. 
\end{equation}
This value of $ Q$   is 
\begin{equation*}
Q \frac{\log Q} {\log\log Q}  \simeq  \bigl[ \langle  f\rangle _{I} \langle g \rangle _{I} \bigr] ^{-1/2} . 
\end{equation*}
Since $ e ^{- c \sqrt{n}    } \ll \langle  f\rangle _{I} \langle g \rangle _{I}$, this is an allowed choice of $ Q$.  
And, then, we prove the desired inequality, but only need a single power of logarithm.  

\smallskip 

Assuming GRH, from \eqref{e:always}, we see that the inequality to prove is always true provided 
$  \langle f \rangle _{I} \langle g \rangle _{I} < cN ^{-1/4} $.  Assuming this inequality fails, we follow the same line of reasoning above 
that leads to \eqref{e:chooseQ}.  That value of $ Q$ will be at most $  N ^{1/4}    $, so the proof will complete, to show the bound with 
a single power of the logarithmic term. 

\end{proof}

Turning to the sparse bounds, let us begin with the definitions. 

\begin{definition}\label{d:sparse} 
A collection of intervals $ \mathcal S$ is called \emph{sparse} if to each interval $ I \in \mathcal S$, there is a set $ E_I \subset I$ so that 
$4 \lvert  E_I\rvert \geq   \lvert  I\rvert $ and the collection $ \{E_I \;:\; I\in \mathcal S\}$ are pairwise disjoint.  
All intervals will be finite sets of consecutive integers in $ \mathbb Z $.   
\end{definition}

The form of the sparse  bound in Theorem~\ref{t:sparseA}  strongly suggests that one use a recursive method of proof. 
(Which is indeed the common method.)  
To formalize it, we start with the notion of a \emph{linearized} maximal function.  
Namely, to bound the maximal function $ A ^{\ast} f $, it suffices to bound $ A _{\tau (x)} f (x)$, where 
$ \tau \;:\; \mathbb Z \to \{2 ^{n} \;:\; n\in \mathbb N \}$ is a function, taken to realize the supremum.  
The supremum in the definition of $ A ^{\ast} f$ is always attained if $ f$ is finitely supported.  

\begin{definition}\label{d:admissible} Let $ I_0 $ an interval, and let $ f$ be supported on $ 3I_0$. 
A map $ \tau \;:\; I_0 \to \{1, 2, 4 ,\dotsc,  \lvert  I_0\rvert \}$  is said to be \emph{admissible} if 
\begin{equation*}
\sup _{N \geq  \tau (x)} M _{N} f (x)  \leq  10  \langle f \rangle _{3 I_0, 1}. 
\end{equation*}
That is, $ \tau $ is admissible if at all locations $ x$, the averages of $ f$ over scales larger than $ \tau (x)$ are controlled by the global average of $ f$.  
\end{definition}

\begin{lemma}\label{l:admissible} Let $ f$ and $ \tau $ be as in Definition~\ref{d:admissible}. 
Further assume that $ f$ and $ g$ are indicator functions, with $ g$ supported on $ I_0$. 
Then, we have 
\begin{equation}\label{e:admissible}
\lvert  I_0\rvert ^{-1}  \langle  A _{\tau } f , g\rangle \lesssim   \langle f \rangle _{I_0,1} \langle g \rangle _{I_0,1} \cdot 
(\operatorname {Log}  \langle f \rangle _{3I_0,1} \langle g \rangle _{I_0,1}) ^{t}, 
\end{equation}
where $ t =1$ assuming RH, and $ t=2$ otherwise. 
\end{lemma}

\begin{proof}
We restrict $ \tau $ to take values $ 1 , 2, 4 ,\dotsc, 2 ^{t} ,\dotsc, $. 
Let $ \lvert  I_0\rvert = N_0 = 2 ^{n_0} $.   
We always have the inequalities 
\begin{align*}
\lvert  I_0\rvert ^{-1}  \langle  A _{\tau } f , g\rangle  
 &\lesssim   n_0 \langle f \rangle _{I_0,1} \langle g \rangle _{I_0,1} 
\\
\lvert  I_0\rvert ^{-1}  \langle   \mathbf 1_{  \tau  <  T} A _{\tau } f , g\rangle  
&\lesssim   (\log T ) \langle f \rangle _{I_0,1} \langle g \rangle _{I_0,1} .  
\end{align*}
The top line follows from admissibility.

We begin by not assuming GRH.  Then, the conclusion of the Lemma is immediate if we have 
$  (\Log \langle f \rangle _{I_0,1} \langle g \rangle _{I_0,1}) ^2  \gg   {n_0}  $.  
It is also immediate if $  \log   \tau  \ll    ( \Log \langle f \rangle _{I_0,1} \langle g \rangle _{I_0,1} ) ^2  $.  
We proceed assuming 
\begin{equation}\label{e:Assuming}
p_0 ^2 = C (\Log \langle f \rangle _{I_0,1} \langle g \rangle _{I_0,1}) ^2 \leq    c_0  \min\{n_0 , \log \tau \} , 
\end{equation}
where $ 0 < c_0 < 1$ is sufficiently small.    

We use the definitions in \eqref{e:Lo}---\eqref{e:Ex} for a value of $ Q < e ^{c \sqrt{    n_0}}$ that we will specify.  
We address the High, Low, Exceptional and both Error terms.    First, the Error terms.  
From the estimate \eqref{e:ErrLess} and \eqref{e:Assuming}, we have 
\begin{align*}
\lVert   \operatorname{Err} _{Q, \tau } f  \rVert_2  ^2 
& \leq  \sum_{ n \;:\;  p_0 ^2  \leq n \leq n_0 }   \lVert   \operatorname{Err} _{Q, 2 ^{n} } f  \rVert_ {\ell ^2 }  ^2 
\\
& \lesssim \lVert f\rVert _{\ell ^2 }  ^2   \sum_{ n \;:\;  p_0 ^2  \leq n \leq n_0 } e ^{- c \sqrt n }
\\ &
\lesssim 
 \lVert f\rVert _{\ell ^2 }  ^2  \cdot   p_0 ^2 e ^{-c p_0} \lesssim  \lVert f\rVert _{\ell ^2 }  ^2  \cdot \langle f \rangle _{3I_0,1} \langle g \rangle_{I_0,1}. 
\end{align*}
This provided $ C$ in \eqref{e:Assuming} is large enough.  This is a much smaller estimate than we need. 
The second error term in Proposition~\ref{p:Err'} is addressed by the same square function argument. 

For the High term, apply \eqref{e:HiSup} to see that 
\begin{equation} \label{e:xHi}
\lVert \sup _{N > Q ^2 }  \lvert   \operatorname{Hi} _{Q,N} f\rvert \rVert_2  \lesssim \frac{ \log Q \cdot \log\log Q} Q  \lVert f\rVert _{\ell ^2 }.  
\end{equation}
For the Low term  the definition of admissibility and \eqref{e:LoLess} that 
\begin{equation*}
\lvert  I_0\rvert ^{-1}    \lvert  \langle   \operatorname{Lo} _{Q, \tau (x)} f (x) , g \rangle \ll  (\log Q)   \langle f \rangle _{I}  \langle g \rangle _{I} . 
\end{equation*}
The Exceptional term also satisfies this bound.

We conclude that 
\begin{align*}
\lvert  I_0\rvert ^{-1}  \langle  A _{\tau } f , g\rangle \lesssim   
 \frac{ \log Q \cdot \log\log Q} Q   \langle f \rangle _{I,2} \langle g \rangle _{I,2} 
 + 
  \log Q  \cdot   \langle f \rangle _{I} \langle g \rangle _{I}. 
\end{align*}
This is optimized by taking $ Q$ so that 
\begin{equation*}
\frac{Q} {\log\log Q} \simeq \bigl[  \langle f \rangle _{I} \langle g \rangle _{I} \bigr] ^{-1/2}. 
\end{equation*}
And this will be an allowed value of $ Q$ since \eqref{e:Assuming} holds.  Again, the resulting estimate is better by 
power of the logarithmic term than what is claimed.  

\smallskip 
Under RH, the proof is very similar, but a wider range of $ Q$'s are allowed.  In particular, only a single power of logarithm is needed.  

\end{proof}

\section{Proof of Corollary~\ref{c:logell}} \label{s:corollary}

The inequality \eqref{e:p-sparse} follows from the elementary identity that for $ 0< x < 1$, we have 
\begin{equation*}
 x (\Log x ) ^{t} \ll  \min _{1< p < 2}   \frac{x } { (p-1) ^{t}} .  
 \end{equation*}
 We remark that we do not know an efficient way to pass from the restricted weak type sparse bound we have established 
 to the similar sparse bounds for functions.  The methods to do this for \emph{norm estimates} is of course very well studied.

\begin{proof}[Proof of \eqref{e:RWT}] 
There is a different  inequality that is a natural consequence of the sparse bound, namely 
\begin{equation}\label{e:different}
\sup _{\lambda } \lambda  \frac{ \lvert  \{A ^{\ast} \mathbf 1_{F} > \lambda \}\rvert } {  (\Log \lvert  \{A ^{\ast} \mathbf 1_{F} > \lambda \}\rvert  \cdot 
\lvert  F\rvert ^{-1}  )} \lesssim \lvert  F\rvert.   
\end{equation}
Indeed, if \eqref{e:RWT} were to fail, with a sufficiently large constant, it would contradict the inequality above.  

Let  $ \lvert  G\rvert > \lvert  F\rvert  $. We show that there 
is a subset $ G' \subset G$, with $  4 \lvert  G'\rvert \geq \lvert G\rvert $ with 
\begin{equation}\label{e:major}
\langle A ^{\ast} f, \mathbf 1_{G'} \rangle \ll  \lvert  F\rvert  (\Log \lvert  F\rvert / \lvert  G\rvert ) ^{t}
\end{equation}
This implies  \eqref{e:different} by taking $ G = \{ A ^{\ast} f > \lambda \}$, for $ 0< \lambda < 1$.  

In the opposite case,  take $ G'$ to be 
\begin{equation*}
G' = G \setminus \{ M f >  K  \rho   \}, \qquad  \rho = \lvert  F\rvert \cdot   \lvert  G\rvert  ^{-1} 
\end{equation*}
where $ M$ is the ordinary maximal function. 
By the usual weak $ \ell ^{1}$ inequality for $ M$, for $ K$ sufficiently large, we have $ 4\lvert  G'\rvert > \lvert  G\rvert  $.  
  Let $ g = \mathbf 1_{G'}$. 
Apply the sparse bound for $ A ^{\ast} $ to see that 
\begin{equation*}
 \langle A ^{\ast} f  , g \rangle 
\ll  \sum_{I\in \mathcal S} \langle  f \rangle _{I}  \langle g \rangle_I (\Log \langle f \rangle_I \langle g \rangle_I) ^{t}  \lvert  I\rvert.  
\end{equation*}
We can assume that for all intervals $ I \in \mathcal S$, that we have $ \langle g \rangle_I >0$. That means that $ \langle f \rangle_I  \leq K \lvert  F\rvert /   \lvert  G\rvert $. 
Turn to a pigeonhole argument.  
Divide the collection $ \mathcal S$ into subcollections $ \bigcup _{j,k\geq 0} \mathcal S _{j,k}$ where 
\begin{equation*}
 \mathcal S _{j,k} = \{I\in \mathcal S \;:\;    2 ^{-j-1}K \rho  < \langle f \rangle_I \leq 2 ^{-j}K \rho  ,\ 
 2 ^{-k-1} < \langle g \rangle_I \leq 2 ^{-k}
 \}. 
\end{equation*}
Then, we have 
\begin{align*}
\langle A ^{\ast} f  , g \rangle  
& \ll \sum_{j,k\geq 0}    
 \sum_{I\in \mathcal S _{j,k}} \langle  f \rangle _{I}  \langle g \rangle_I (\Log \langle f \rangle_I \langle g \rangle_I) ^{t}  \lvert  I\rvert
 \\
 &\ll  \lvert  F\rvert  \cdot \lvert  G\rvert ^{-1}  
 \sum_{j,k\geq 0}     2 ^{-j-k} (j+k +  \Log \rho  ) ^{t}   \sum_{I\in \mathcal S _{j,k}} \lvert  I\rvert  
 \\
 & \ll  
  \lvert  F\rvert  \cdot \lvert  G\rvert ^{-1}  
 \sum_{j,k\geq 0}     2 ^{-j-k} (j+k  +\Log \rho  ) ^{t}  \min \{ \lvert  G\rvert 2 ^{j} ,     2 ^{k} \lvert  G\rvert \}  
 \\
 & \ll  \lvert  F\rvert  \sum_{j,k\geq 0}     2 ^{-j-k} (j+k+\Log \rho )  2 ^{ (j+k)/2} \ll  \lvert  F\rvert .  
\end{align*}
Here, we have used the standard weak-type inequality for the maximal function, and the basic property of sparseness, namely 
\begin{equation*}
 \sum_{I\in \mathcal S} \lvert  I\rvert \lesssim \Bigl\lvert  \bigcup _{I\in \mathcal S} I \Bigr\rvert .  
\end{equation*}
This completes the proof of \eqref{e:major}.  
\end{proof}

For the proof of \eqref{e:orlicz}, we need to recall the definition of the Orlicz norm.  
Given $ f $ finitely supported on $ \mathbb Z $, let $ f ^{\ast} \;:\; [0, \infty ) \to \mathbb N  $ be the decreasing rearrangement of $ f$. 
That is, 
\begin{equation*}
f ^{\ast} (\lambda ) = \lvert  \{ x \in \mathbb Z \;:\; \lvert  f (x)\rvert \geq \lambda \}\rvert.  
\end{equation*}
For a  slowly varying  function $ \varphi  \;:\; [0, \infty ) \to [0, \infty )$, set 
\begin{align} 
\lVert f\rVert _{  \ell \varphi (\ell )}  &=  \int _{0} ^{\infty }  f ^{\ast} (\lambda ) \varphi (\lambda )  \; d \lambda  
\\  \label{e:ell-Log}
&  \simeq \sum_{j\in \mathbb Z }   2 ^{j} \varphi ( 2 ^{j})  f ^{\ast} (2 ^{j}).  
\end{align}
For $ \varphi (x) = 1$, this is comparable to  the usual $ \ell ^{1}$ estimate.  For $ f = \mathbf 1_{F}$, note that 
\begin{equation*}
\lVert f\rVert _{ \ell \varphi (\ell )}  
= \int _{0} ^{\lvert  F\rvert }     \varphi (\lambda  ) \; d \lambda   \simeq  \lvert  F\rvert \varphi (\lvert  F\rvert ) 
\end{equation*}
We are interested in $ \varphi (x) =   ( \Log  x)  \cdot  \Log  \Log x)^{t} $, for $ t=1,2$.  
The proof of the orlicz norm estimate \eqref{e:orlicz} is below. 

\begin{proof}[Proof of \eqref{e:orlicz}] 
 This argument goes back to at least \cite{MR241885}.  
Assume that the weak-type estimate for indicators \eqref{e:RWT} holds. 
Let $ f \in \ell (\log \ell ) ^{t} (\log\log \ell )$ be a non-negative function of norm one.   Set 
\begin{gather*}
  B_j = \{ x \;:\;   2 ^{j } \leq f (x) < 2 ^{j+1}  \}, 
\end{gather*}
and set $  b_j = f ^{\ast} (2 ^{j})$.   We have 
\begin{equation*}
\sum_{j \leq 0} 2 ^{j} \mathbf 1_{B_j} \leq f \leq 2 \sum_{j \leq 0} 2 ^{j} \mathbf 1_{B_j} . 
\end{equation*}
And, by logarithmic subadditivity for the weak-type norm,  and \eqref{e:RWT}, 
\begin{align*}
\lVert  A ^{\ast} f \rVert _{1, \infty } & \ll \sum_{j \leq 0}  \log (1-j)  \cdot   2 ^{j} \lVert  A ^{\ast}  \mathbf 1_{B_j} \rVert _{1, \infty }  
\\
& \ll  \sum_{j \leq 0}  \log (1-j)  \cdot 2 ^{j}  \lvert  B_j\rvert  (\log \lvert  B_j\rvert ) ^{t} 
\\
& \ll   \sum_{j \leq 0}  \log (1-j)  \cdot  j ^{t}  2 ^{j}  \lvert  B_j\rvert   \ll  \lVert f\rVert _{ \ell (\log \ell ) ^{t} (\log\log \ell )} =1. 
\end{align*}
Above, we appealed to $   \lvert  B_j \rvert  \leq 2 ^{-j} $, for otherwise the norm of $ f$ is more than one.

\end{proof}


\bibliographystyle{alpha,amsplain}


\end{document}